\documentclass[12pt]{article}

\def\UseBibLatex{1}

\makeatletter
\def\input@path{{styles/}}
\makeatother

\newcommand{\UsePackage}[1]{%
  \IfFileExists{styles/#1.sty}{%
      \usepackage{styles/#1}%
   }{%
      \IfFileExists{../styles/#1.sty}{%
         \usepackage{../styles/#1}%
      }{%
         \usepackage{#1}%
      }%
   }%
}

\usepackage[T1]{fontenc}
\usepackage{lmodern}
\usepackage{textcomp}

\usepackage{amsmath}%
\usepackage{amssymb}%
\usepackage[table]{xcolor}%

\usepackage{todonotes}
\usepackage[in]{fullpage}%

\usepackage[amsmath,thmmarks]{ntheorem}%
\theoremseparator{.}%

\usepackage{titlesec}%
\titlelabel{\thetitle. }%
\usepackage{xcolor}%
\usepackage{mleftright}%
\usepackage{xspace}%
\usepackage{graphicx}
\usepackage{hyperref}%
\usepackage{multirow}
\usepackage{amsmath,amssymb}
\usepackage{indentfirst}
\usepackage{amsmath}
\usepackage{tabularx}

\usepackage{hyperref}%
\hypersetup{%
      unicode,
      breaklinks,%
      colorlinks=true,%
      urlcolor=[rgb]{0.25,0.0,0.0},%
      linkcolor=[rgb]{0.5,0.0,0.0},%
      citecolor=[rgb]{0,0.2,0.445},%
      filecolor=[rgb]{0,0,0.4},
      anchorcolor=[rgb]={0.0,0.1,0.2}%
}
\usepackage[ocgcolorlinks]{ocgx2}

\theoremseparator{.}%

\theoremstyle{plain}%
\newtheorem{theorem}{Theorem}[section]

\newtheorem{proposition}[theorem]{Proposition}

\theoremstyle{plain}%
\newtheorem*{remark:unnumbered}[theorem]{Remark}%
\newtheorem{remark}[theorem]{Remark}%
\newtheorem{definition}[theorem]{Definition}

\newtheorem{example}[theorem]{Example}

\newcommand{\myqedsymbol}{\rule{2mm}{2mm}}

\theoremheaderfont{\em}%
\theorembodyfont{\upshape}%
\theoremstyle{nonumberplain}%
\theoremseparator{}%
\theoremsymbol{\myqedsymbol}%
\newtheorem{proof}{Proof:}%

\newcommand{\ben}{\begin{enumerate}}
\newcommand{\een}{\end{enumerate}}

\newcommand{\Z}{\mathbb{Z}}

\newcommand{\A}{\text{Aut}}
\newcommand{\I}{\text{Inn}}
\newcommand{\T}{\text{Transv}}

\providecommand{\emphind}[1]{}%
\renewcommand{\emphind}[1]{\emph{#1}\index{#1}}

\definecolor{blue25emph}{rgb}{0, 0, 11}

\providecommand{\emphic}[2]{}
\renewcommand{\emphic}[2]{\textcolor{blue25emph}{%
      \textbf{\emph{#1}}}\index{#2}}

\providecommand{\emphi}[1]{}%
\renewcommand{\emphi}[1]{\emphic{#1}{#1}}

\definecolor{almostblack}{rgb}{0, 0, 0.3}

\providecommand{\emphw}[1]{}%
\renewcommand{\emphw}[1]{{\textcolor{almostblack}{\emph{#1}}}}%

\providecommand{\emphOnly}[1]{}%
\renewcommand{\emphOnly}[1]{\emph{\textcolor{blue25}{\textbf{#1}}}}

\newcommand{\QuinnThanks}[1]{%
   \thanks{%
      Department of Mathematics; %
      University of California, Davis; %
      One Shields Ave; %
      Davis, CA, 95616, USA; %
      \href{mailto:qjmcmichael@ucdavis.edu}{qjmcmichael@ucdavis.edu}; %
      \url{}. %
   #1%
   }%
}

\newcommand{\PeytonThanks}[1]{%
   \thanks{%
      Department of Mathematics; %
      University of California, Davis; %
      One Shields Ave; %
      Davis, CA, 95616, USA; %
      \href{mailto:ppwood@ucdavis.edu}{ppwood@ucdavis.edu}; %
      \url{https://sites.google.com/view/peytonpwood/}. %
   #1%
   }%
}

\newcommand{\HLink}[2]{\hyperref[#2]{#1~\ref*{#2}}}
\newcommand{\HLinkSuffix}[3]{\hyperref[#2]{#1\ref*{#2}{#3}}}

\providecommand{\deflab}[1]{}
\renewcommand{\deflab}[1]{\label{def:#1}}

\providecommand{\eqlab}[1]{}%
\renewcommand{\eqlab}[1]{\label{equation:#1}}

\newcommand{\remove}[1]{}%

\usepackage[inline]{enumitem}

\newlist{compactenumA}{enumerate}{5}%
\setlist[compactenumA]{topsep=0pt,itemsep=-1ex,partopsep=1ex,parsep=1ex,%
   label=(\Alph*)}%

\newlist{compactenuma}{enumerate}{5}%
\setlist[compactenuma]{topsep=0pt,itemsep=-1ex,partopsep=1ex,parsep=1ex,%
   label=(\alph*)}%

\newlist{compactenumI}{enumerate}{5}%
\setlist[compactenumI]{topsep=0pt,itemsep=-1ex,partopsep=1ex,parsep=1ex,%
   label=(\Roman*)}%

\newlist{compactenumi}{enumerate}{5}%
\setlist[compactenumi]{topsep=0pt,itemsep=-1ex,partopsep=1ex,parsep=1ex,%
   label=(\roman*)}%

\newlist{compactitem}{itemize}{5}%
\setlist[compactitem]{topsep=0pt,itemsep=-1ex,partopsep=1ex,parsep=1ex,%
   label=\ensuremath{\bullet}}%

\providecommand{\BibLatexMode}[1]{}
\providecommand{\BibTexMode}[1]{}

\ifx\UseBibLatex\undefined%
  \renewcommand{\BibLatexMode}[1]{}
  \renewcommand{\BibTexMode}[1]{#1}
\else
  \renewcommand{\BibLatexMode}[1]{#1}
  \renewcommand{\BibTexMode}[1]{}
\fi

\BibLatexMode{%
   \usepackage[bibencoding=utf8,style=alphabetic,backend=biber]{biblatex}%
   \UsePackage{my_biblatex}%
}

\numberwithin{figure}{section}%
\numberwithin{table}{section}%
\numberwithin{equation}{section}%

\BibLatexMode{%
   \bibliography{template}
}

\begin{document}

\title{More Automorphism Groups of Quandles}

\author{%
   Quinn J. M. Salix
   \QuinnThanks{}%
   \and%
   Peyton Phinehas Wood%
   \PeytonThanks{}%
}

\date{September 27, 2025}

\maketitle

\begin{abstract}
    We prove that the displacement group of the dihedral quandle with $n$ elements is isomorphic to the group generated by rotations of the $\frac{n}{2}$-gon when $n$ is even and the $n$-gon when $n$ is odd. We additionally show that any quandle with at least one trivial column has equivalent displacement and inner automorphism groups. Then, using a known enumeration of quandles which we confirm up to order 10, we verify the automorphism group and the inner automorphism group of all quandles (up to isomorphism) of orders less than or equal to 7, compute these for all 115,431 quandles orders 8, 9, and 10, and extend these results by computing the displacement group of all 115,837 quandles (up to isomorphism) of order $\le$ 10.
\end{abstract}

    KEYWORDS: quandle, automorphism group, inner automorphism group, displacement group, transvection group, dihedral quandle, dual quandle


\section{Introduction}

    Joyce and Matveev both independently introduced quandles to the mathematical world in the 1980s \cite{Joyce_1982, 1984_Matveev}. The word ``quandle'' was first coined by David Joyce in his 1982 thesis  whereas Matveev's terminology of ``distributive groupoid'' for the same object has not been used as widely.

Quandles are a useful object to employ when studying knots such as in \cite{Carter1999QuandleCA}, \cite{Carter_Jelsovsky_Kamada_Saito_2001}, \cite{Clark_Elhamdadi_Saito_Yeatman_2014}, \cite{Elhamdadi_Nelson_2015}, \cite{Elhamdadi_Nunez_Singh_Swain_2023}, and \cite{Swain_2024}, to name a few, but their algebraic properties are not as often highlighted. The transvection, inner automorphism, and automorphism groups of a quandle were first defined by Joyce in \cite{Joyce_1982}, but they have only begun being systematically computed for finite quandles in the last twenty years. This first group, the transvection group, is now popularly referred to as the \textit{displacement group}, first coined by Hulpke, Stanovský, and Vojtěchovský in \cite{HULPKE2016735}. We will use this convention throughout. In 2005, Ho and Nelson computed the automorphism group of all quandles (up to isomorphism) of order $\le$ 5. In 2010, this was extended to $\le$ 6 by Elhamdadi, MacQuarrie, and Restrepo \cite{AutGroup2010} and again by MacQuarrie to $\le$ 7 in her 2011 thesis \cite{MacQuarrie_2011}. MacQuarrie also computed all inner automorphisms for finite quandles (up to isomorphism) of order $\le$ 7.

In 2019, Vojtěchovský and Yang enumerated all quandles (up to isomorphism) of order $\le$ 13 \cite{Vor_Yang_2019} which we have reconfirmed for orders $\le$ 10. The problem of computing all quandles up to isomorphism is deeply related to computing the automorphism groups of these quandles, so traditionally both have been done together as in \cite{Mat_and_fin_q_ho_nelson_2005} and \cite{AutGroup2010}; however, Vojtěchovský and Yang only gave us the quandles, not their automorphism groups, leaving computing the automorphism groups of quandles of orders 8, 9, 10, 11, 12 and 13 open. In this paper, we compute the automorphism and inner automorphism groups for all 1,581 quandles of order 8, 11,079 quandles of order 9, and 102,771 quandles of order 10.

Nelson and Ho additionally give the automorphism group of the trivial quandle \cite{Mat_and_fin_q_ho_nelson_2005}. In this paper, we give the trivial quandle's inner automorphism and displacement groups, having not found these results already in existing literature. Elhamdadi, MacQuarrie, and Restrepo additionally characterized the automorphism and inner automorphism group for the dihedral quandle \cite{AutGroup2010}.  In this paper, we prove the displacement group of the dihedral quandle with $n$ elements is isomorphic to the group generated by rotations of the $\frac{n}{2}$-gon when $n$ is even and the $n$-gon when $n$ is odd. Cazet proved a characterization for the inner and automorphism groups of one non-trivial column quandles in \cite{cazet2024quandlesnontrivialcolumn}. In this paper, we show the displacement group of one non-trivial column quandles is equal to its inner automorphism group, extending Cazet's result. Furthermore, we show that for all quandles with at least one trivial column the inner automorphism and displacement group are equal. 

In Section 2 we review common quandle theoretic definitions especially the displacement (or transvection) group, inner automorphism group, and automorphism group of a quandle, then provide several examples of families of quandles that we explore later in this paper. In Section 3 we highlight several known results of other authors that are relevant to the questions we answer.

In Section 4 we share several new results concerning the displacement, inner automorphism, and automorphism groups of quandles. To the authors of this paper's knowledge this is the first paper giving significant attention to computing the displacement group of small finite quandles. Theorem \ref{MainResult}, fully classifies the displacement group of dihedral quandles.

In Section 5 we explain how we compute the displacement group for all quandles (up to isomorphism) of order $\le$ 10; reconfirm the results of \cite{Mat_and_fin_q_ho_nelson_2005}, \cite{AutGroup2010}, and \cite{MacQuarrie_2011}; and extend these results to orders 8, 9 and 10. We summarize some of these computations in a series of tables in Appendix A.


\section{Review of Quandles}

\begin{definition}[Quandle, \cite{Elhamdadi_Nelson_2015}]
    A {\em quandle}, $(X,\triangleright)$, is a set X with binary operation $\triangleright : X \times X \to X$ satisfying:
    \begin{enumerate}
        \item Q1 (Idempotency): For all x $\in$ X, $x \triangleright x = x$.
        \item Q2 (Invertibility): For all y $\in$ X, the map $\beta_y : X \to X$ defined by $\beta_y(x)=x\triangleright y$ is invertible. We denote $\beta^{-1}_y(x)$ as $x \triangleright^{-1} y$. 
        \item Q3 (Self-Distributivity): For all x, y, z $\in$ X, $(x \triangleright y) \triangleright z = (x \triangleright z) \triangleright (y \triangleright z) $.
    \end{enumerate}
    \end{definition}

    Q1, Q2, and Q3 are known as the {\em quandle axioms}. A quandle is called \textit{finite} if $|X| \in \mathbb{Z}$. We say $n=|X|$ is the \textit{order} of a quandle. Finite quandles have been enumerated (up to isomorphism) for order $n \le$ 13 \cite{Vor_Yang_2019}.

\begin{definition}[Presentation Matrix of a Finite Quandle]
Let $(X,\triangleright)$ be a quandle of order $n$. The $n \times n $ matrix
\[ M_X = \begin{bmatrix} x_1 \triangleright x_1 & x_1 \triangleright x_2& \dots &x_1 \triangleright x_3\\
x_2 \triangleright x_1 & x_2 \triangleright x_2& \dots &x_1 \triangleright x_n\\
\vdots & \vdots & \ddots & \vdots\\
x_n \triangleright x_1 & x_n \triangleright x_2& \dots & x_n \triangleright x_n \end{bmatrix}\]
is called the \textit{presentation matrix of X}.
    
\end{definition}

We choose to denote 0 in the following quandles with $n$.

\begin{example}
    The \textit{Trivial Quandle of order n}, denoted $T_n$: For any $n\in\Z$, let $X=\mathbb{Z}_n $ and $\forall x, y \in X$, $x \triangleright y = x$ .
\end{example}

\begin{figure}[h]
    \centering
    $\begin{bmatrix}1&1\\2&2\end{bmatrix}$
    \hspace{.5in}
    $\begin{bmatrix}1&1&1\\2&2&2\\3&3&3\end{bmatrix}$
    \caption{Presentation Matrices for $T_2$ (left) and $T_3$ (right) }
\end{figure}

\begin{example}
    The \textit{Dihedral Quandle of order n}, denoted $R_n$: For any $n\in\Z$, let $X=\mathbb{Z}_n $ and $\forall x, y \in X$, $x \triangleright y = 2y-x \mod{n}$. 
\end{example}

\begin{figure}[h]
    \centering
    $\begin{bmatrix}1&3&2\\3&2&1\\2&1&3\end{bmatrix}$
    \hspace{.5in}
    $\begin{bmatrix}1&3&1&3\\4&2&4&2\\3&1&3&1\\2&4&2&4\end{bmatrix}$
    \caption{Presentation Matrices for $R_3$ (left) and $R_4$ (right)}
\end{figure}

\begin{example}
    An \textit{Alexander Quandle of order n}: Pick $n \in \Z$ and choose $t \in \Z_n$ such that $\gcd(t,n)=1$. Then $x \triangleright y = (1-t)y+tx$ is an order n quandle.
\end{example}

\begin{example}
    The \textit{conjugation quandle of $G$}, denoted Conj($G$): Let $G$ be any group, then $(G,\triangleright)$ is a quandle when $\forall a, b \in G$, $a \triangleright b = bab^{-1}$. 
\end{example}

\begin{example}
    The \textit{core quandle of G}, denoted Core($G$): Let $G$ be any group, then $(G,\triangleright)$ is a quandle when $\forall a, b \in G$, $a \triangleright b = ba^{-1}b$. 
\end{example}

\begin{example}
    The \textit{Takasaki quandle of A}, denoted $T(A)$: Let $A$ be any additive abelian group, then $(A,\triangleright)$ is a quandle when $\forall a, b \in A$, $a \triangleright b = 2b - a $.
\end{example}

\begin{remark}
    When $A$ is abelian, Core($A$) = $T(A)$. Additionally, choosing $A=\Z_n$ for Takasaki quandles yields the dihedral quandle, $R_n$.
\end{remark}

\begin{example}
    A one non-trivial column quandle, denoted $P_{n}^{\sigma}$: Let X=$\Z_n$ and let $\sigma \in S_n$ such that $\sigma(0)=0$. Then the function 

    $$ x \triangleright y = 
    \begin{cases}
    \sigma(x) & y=0\\
    x & y \ne 0\\
    \end{cases}$$

    is a quandle.
    
\end{example}

\begin{figure}[h]
    \centering
    $\begin{bmatrix}1&1&1\\3&2&2\\2&3&3\end{bmatrix}$
    \hspace{.5in}
    $\begin{bmatrix}1&1&1&1\\4&2&2&2\\2&3&3&3\\3&4&4&4\end{bmatrix}$
    \caption{Presentation Matrices for $P_3^{(23)}$ (left) and $P_4^{(432)}$ (right)}
\end{figure}

A \textit{quandle homomorphism} is a function $f : (X, \triangleright_1) \to (Y, \triangleright_2)$ such that $f(x \triangleright_1 y) = f(x) \triangleright_2 f(y)$. We also have the expected notion of a quandle isomorphism and a quandle automorphism \cite{Elhamdadi_Nelson_2015}.

\begin{example}
    The map $\beta_z$ for $z \in X$ defined by Q2 is a quandle homomorphism due to Q3. We can write Q3 in terms of $\beta$'s as follows: $\beta_z(x \triangleright y) = \beta_z(x) \triangleright \beta_z(y)$. Furthermore, it is also a quandle isomorphism and automorphism since Q2 gives that $\beta_z$ is a bijection from X to itself.
\end{example}

\begin{definition}[Automorphism Group of a Quandle]
    Let Aut($X$) denote the group of all automorphisms of the quandle $X$. That is, \[ \text{Aut}(X) \cong \{\sigma \in S_n: \sigma(X) = X\}.\]
\end{definition}

\begin{definition}[Inner Automorphism Group of a Quandle]
    Let Inn($X$) denote the group of all inner automorphisms of the quandle $X$. This is the group generated by permutations corresponding to $\beta_y$ for all $y \in X$ since Q1 requires each $\beta_y$ fix at least one point. That is, Inn($X$) = $\langle \beta_y : y \in X \rangle$.
\end{definition}

\begin{definition}[Displacement Group of a Quandle, \cite{Joyce_1982}]
    Let Dis($X$) or Transv($X$) denote the group of automorphisms of the form $\beta_{x_1}^{e_1}\beta_{x_2}^{e_2} \cdots \beta_{x_n}^{e_n}$ such that $e_1 + e_2 + \cdots + e_n = 0$. This is the group generated by  permutations corresponding to $\beta_x\beta_y^{-1}$ for all $x, y \in X$. That is, Dis($X$) = $\langle \beta_x\beta_y^{-1} : x, y \in X \rangle$.
\end{definition}

\begin{definition}[Affine Group of $\Z_n$]
    The affine group consists of all functions of the form $ax+b$ where $a,b \in \Z_n$ and $\gcd(a,n)=1$. These represent all affine transformations over $\Z_n$.
\end{definition}


\section{Some Known Results Involving Dis(X), Inn(X), and Aut(X)}

Throughout these results, let $X$ be a quandle (not necessarily finite), let $n$ be the order of $X$ when $X$ is finite, let $T_n$ denote the trivial quandle of order $n$, and let $R_n$ denote the dihedral quandle of order $n$.

\begin{proposition}
    When $X$ is finite, Dis($X$), Inn($X$), and Aut($X$) are all isomorphic to subgroups of the symmetric group, $S_n$.
\end{proposition}

The next three results were given by Joyce without proof in his PhD thesis in a section he devotes to groups of automorphisms on quandles.

\begin{proposition}[Section 5 in \cite{Joyce_1982}]
    Inn($X$) is a normal subgroup of Aut($X$).
\end{proposition}

\begin{proposition}[Section 5 in \cite{Joyce_1982}]
    Dis($X$) is a normal subgroup of Aut($X$).
\end{proposition}

\begin{proposition}[Section 5 in \cite{Joyce_1982}]\label{prop:transv_subset_inn}
    Dis($X$) is a normal subgroup of Inn($X$).
\end{proposition}

One popular idea is to classify $\T(X), \I(X), \text{and } \A(X)$ for known types of quandles. For example, Nelson and Ho give the below classification of the automorphism group of the trivial quandle.

\begin{proposition}[Example 1 in \cite{Mat_and_fin_q_ho_nelson_2005}]
    The automorphism group of $T_n$ is isomorphic to $S_n$.
\end{proposition}

In a 2010 paper and 2011 Master thesis, MacQuarrie, Elhamdadi, and Restrepo give two results for Dihedral quandles and one for Conjugation quandles.

\begin{theorem}[Thm 2.1 in \cite{AutGroup2010}]
    The automorphism group of $R_n$ is isomorphic to the affine group Aff($\Z_n$).
\end{theorem}

\begin{theorem}[Thm 2.3 in \cite{AutGroup2010}]
    The inner automorphism group of $R_n$ is isomorphic to the dihedral group $D_{\frac{n}{2}}$ when $n$ is even and $D_n$ when $n$ is odd.
\end{theorem}

\begin{theorem}[Thm 2.4 in \cite{AutGroup2010}]
    The inner automorphism group of the conjugation quandle of a group, Inn(Conj($G$)) is isomorphic to $G / Z(G)$ where $Z(G)$ is the center of $G$.
\end{theorem}

In 2017, Bardakov, Dey, and Singh generalize MacQuarrie, Elhamdadi and Restrepo's results for dihedral quandles to Takasaki quandles on albelian groups with no 2-torsion. They further prove some stronger results that are not reproduced below.

\begin{theorem}[Thm 4.2 in \cite{BDS_2017_Aut_Conj}]
    If $G$ is an additive abelian group with no 2-torsion, then:
    \begin{enumerate}
        \item $\text{Aut}(T(G)) \cong G \rtimes \text{Aut}(G)$
        \item $\I(T(G)) \cong 2G \rtimes \Z_2$
        
    \end{enumerate}
\end{theorem}

In 2019, Bardakov and Singh furthered these results about the Automorphism groups of conjugation quandles with Nasybullov.

\begin{theorem}[Thm 4.8 in \cite{BNS_2019_Aut_Conj}]
    For a finite group $G$, $\A(\text{Conj}(G))) =\A(G) \times S_{|Z(G)|} $ if and only if $Z(G)=1$ or $G=\Z_2$.
\end{theorem}

\begin{theorem}[Thm 4.9 in \cite{BNS_2019_Aut_Conj}]
    For a finite group $G$, $\A(\text{Conj}(G))) = Z(G) \rtimes \A(G)$ if and only if either $Z(G) = 1$ or $G = \Z_2, \Z_2^2, \text{ or } \Z_3$.
\end{theorem}

In 2024, Cazet posted a pre-print to arXiv which will appear in the \textit{Journal of Algebra and Its Applications} concerning various aspects of one non-trivial column quandles. He includes a result about their automorphism and inner automorphism groups.

\begin{theorem}[Thm 3.4 in \cite{cazet2024quandlesnontrivialcolumn}]
    If $\sigma \ne ()$, then $\text{Aut}(P_{n}^{\sigma}) \cong C_{S_n}(\sigma)$ and $\text{Inn}(P_{n}^{\sigma}) \cong \Z_{|\sigma|}$ where $C_{S_n}(\sigma)$ is the centralizer of $\sigma$ and $|\sigma|$ is the order of $\sigma$.
\end{theorem}


\section{New Results Concerning the Displacement, Inner Automorphism, and Automorphism Groups of Quandles}

Although the following proposition is trivial and generally known, it has yet to be explicitly stated in the literature, so we produce it here.
\begin{proposition}
    For the trivial quandle $T_n$, $\text{Inn}(T_n) \cong \{1\}$, the trivial group. Furthermore, $\text{Dis}(T_n) \cong \{1\}$ as well.
\end{proposition}

\begin{proof}
As $T_n$ is trivial, $\beta_i$ is the identity permutation for all $i\in \{1,2,\ldots,n\}$. Thus $\text{Inn}(T_n)$, generated by the set of $\beta_i$, is trivial.
\par Because $\text{Dis}(T_n)\subseteq \text{Inn}(T_n)$, it follows that $\text{Dis}(T_n)$ is also trivial.
\end{proof}

\begin{theorem}\label{MainResult}
    The displacement group of the dihedral quandle, $R_n$, is isomorphic to $Z_n$ when $n$ is odd and $\Z_{\frac{n}{2}}$ when $n$ is even.\\
\end{theorem}

\begin{proof}

    In $R_n$, each $\beta_y$ corresponds to the permutation
    $$(y-1, y+1)(y-2, y+2)(y-2) \cdots (y- \lfloor \tfrac{n}{2} \rfloor, y+ \lfloor \tfrac{n}{2} \rfloor)$$ where everything is taken $\mod{n}$.
    
    Consider then $\beta_i\beta_j$. This sends $ k \to k-2a$ where $k \in [n]$ and $a = j-i \mod{n}$. 
    
    Let $\alpha_a = \beta_i\beta_j$. This is the identity function when $i=j$ and is a permutation with no fixed points when $i\ne j$.

     Let $k, b \in [n]$ so that $\alpha_b(k)=k-2b$ and $(\alpha_b)^p(k)=k-2pb$ for $p \in \Z$. \\

    \noindent
    Case 1: \textit{$n$ is odd}\\
    
    First, choose $a=\frac{n-1}{2}$. Thus, $\alpha_{a}$ corresponds to the permutation $(123\cdots n)$. We are going to show that all $\alpha_b$ correspond to $(12\dots n)^l$ for some $l \in \Z$.\\

    \noindent
    Notice $(\alpha_{\frac{n-1}{2}})^{n-2b}(k) = k - 2(n-2b)(\frac{n-1}{2}) \mod{n}$\\
    \hspace*{1.46in}    = $k - n(n-1-2b) - 2b \mod{n}$\\
     \hspace*{1.46in}   = $k - 2b \mod{n}$\\
     \hspace*{1.46in}   = $\alpha_b(k)$.

    Thus, for all $b \in [n]$, $\alpha_b(k)$ = $(\alpha_{\frac{n-1}{2}})^{n-2b}(k)$, so each $\beta_i\beta_j$ corresponds to the permutation $(12 \dots n)^{n-2b}$. Hence, Dis($R_n$) = $\Z_n$.\\

    \noindent
    Case 2: \textit{$n$ is even}\\
    Note that choosing $j=1$ and $i = \frac{n+4}{2} \pmod{n}$ corresponds to the permutation $(1357 \dots n-1)(2468 \dots n)$. We are going to show that all $\beta_i\beta_j$ correspond to $[(1357 \dots n-1)(2468 \dots n)]^{l}$ for some $l \in \Z$.\\

    \noindent
    Notice $(\alpha_{\frac{n-2}{2}})^{n-b}(k) = k - 2(n-b)(\frac{n-2}{2}) \mod{n}$\\
   \hspace*{1.405in}    = $k - n(n-2-b) - 2b \mod{n}$\\ 
    \hspace*{1.405in}   = $k - 2b \mod{n}$\\
     \hspace*{1.405in}   = $\alpha_b(k)$.

Thus, each $\beta_i\beta_j$ corresponds to the permutation $[(1357 \dots n-1)(2468 \dots n)]^{n-b}$ and $\langle (1357 \dots n-1)(2468 \dots n) \rangle$ is isomorphic to $\Z_{\frac{n}{2}}$.
    
\end{proof}

\begin{proposition}
    The displacement group of a one non-trivial column quandle is the same as its inner automorphism group, that is, $\text{Dis}(P_{n}^{\sigma}) = \text{Inn}(P_{n}^{\sigma})$
\end{proposition}

\begin{proof}
    Let $P_{n}^{\sigma}$ be a one non-trivial column quandle. We know from the definition that $\text{Dis}(P_{n}^{\sigma})$ is generated by $\beta_i\beta_j^{-1}$ where $i,j \in {1, \dots, n}$. There are three cases:\\
    1. If $i,j \ne 0$ or $i,j = 1$, then $\beta_i\beta_j^{-1} = ()$.\\
    2. If $i=1$, $j \ne 1$, then $\beta_i\beta_j^{-1} = \beta_1$.\\
    3. If $i \ne 1$, $j = 1$, then $\beta_i\beta_j^{-1} = \beta_1^{-1}$.\\
    Thus, $\text{Dis}(P_{n}^{\sigma}) = \langle \beta_1 \rangle$ which is isomorphic to $\Z_{|\sigma|}$ and hence, $\text{Inn}(P_{n}^{\sigma})$.
\end{proof}

\begin{theorem}\label{thm:transv_equal_inn}
    If a quandle $X$ has a trivial column, then $\text{Dis}(X) = \text{Inn}(X)$.
\end{theorem}

\begin{proof}
    Let the $j$-th column be trivial. Then for all $i$, $\beta_i \beta_j^{-1} = \beta_i$. That is to say, each $\beta_i$ is included in the set of generators for $\text{Dis}(X)$.\\
    Because $\text{Inn}(X)$ is generated solely by the $\beta_i$, we have that $\text{Dis}(X) \supseteq \text{Inn}(X)$. But by Proposition \ref{prop:transv_subset_inn}, $\text{Dis}(X) \subseteq \text{Inn}(X)$ for any quandle, and we acquire $\text{Dis}(X) = \text{Inn}(X)$.
\end{proof}

\begin{remark}
    Notice that Theorem \ref{thm:transv_equal_inn} holds for infinite quandles, by the same proof.
\end{remark}


\section{Computing the Displacement, Inner Automorphism, and Automorphism Groups of Quandles}

As stated previously, an enumeration of all quandles up to order 13 is provided by \cite{Vor_Yang_2019}. Data for this is provided by Vojtěchovský on his personal website \cite{QuandleData} that encodes each quandle as an ordered list of permutations in cycle notation, corresponding to the columns of the quandle matrix. This data is in the GAP software syntax.

We initially began our computation for Dis($X$), using Octave, then Python, before finally settling on Sage. These computations are deeply connected to finding subgroups of the symmetric group which Sage has an easier time handling.

We reconfirm some results of \cite{AutGroup2010} by calculating the automorphism and inner automorphism groups of all quandles of order $\le$ 5. This is done by first converting the GAP data into a list of permutation group elements in Sage, corresponding to the columns of the quandle matrix. These permutations were used to generate the inner automorphism group using the \verb|PermutationGroup()| function from Sage. The resulting groups were then checked against a large list of groups using the \verb|.is_isomorphic()| Sage function, to identify the isomorphism class of each.

The displacement groups were calculated in a similar fashion to the inner automorphism groups. From the list of permutations representing the quandle matrix columns, elements were composed pairwise to acquire a set of generators for the displacement group, specifically the set $\{ a\;b^{-1} | a,b\in C\}$, for $C$ the set of permutations. This set was used as the generators for the displacement group, and the resulting groups were then identified with their isomorphism class by the same method as above.

To find the automorphism group of each quandle, we first constructed the quandle matrix, using the permutations from the GAP data to create each column. Then, using the \verb|SymmetricGroup()| function in Sage, we checked if each element of the symmetric group was an automorphism by multiplying it to the quandle matrix.

Then, having acquired a list of valid automorphisms, the \verb|PermutationGroup()| Sage function was again used to generate the automorphism group, and the isomorphism class of this group was found in the same manner as above.


\section*{Acknowledgments}

We would like to thank Jennifer Schultens for her help with edits. We are thankful to the many mathematicians who have made packages involving symmetric group subgroup computations available in Sage. We would also like to thank Jennifer Brown for first bringing Quandles to UC Davis.



\BibLatexMode{\printbibliography}

\appendix

\section{Tables}

We give presentation matrices and displacement groups for all quandles of orders 1, 2, 3, 4 and 5.

\begin{center}
\begin{tabular}{ |c|c| } 
\hline
Quandle, X & Dis($X$) \\
\hline
\multirow{3}{3em}{$\begin{bmatrix}
    1 \\
\end{bmatrix}$} & \\ & $\{ 1 \}$ \\ & \\
\hline

\multirow{3}{4em}{$\begin{bmatrix}
    1 & 1\\
    2 & 2\\
\end{bmatrix}$} & \\ & $\{ 1 \}$ \\ & \\
\hline
&\\

\multirow{2}{4em}{$\begin{bmatrix}
    1 & 1 & 1\\
    2 & 2 & 2\\
    3 & 3 & 3\\
\end{bmatrix}$}
& \\ & $\{ 1 \}$ \\ & \\ & \\
\hline
&\\

\multirow{2}{4em}{$\begin{bmatrix}
    1 & 1 & 1 \\
    3 & 2 & 2 \\
    2 & 3 & 3 \\
\end{bmatrix}$} & \\ & $\Z_2$ \\ & \\ & \\
\hline
&\\

\multirow{3}{4em}{$\begin{bmatrix}
    1 & 3 & 2\\
    3 & 2 & 1\\
    2 & 1 & 3 \\
\end{bmatrix}$} & \\ & $\Z_3$ \\ & \\ & \\
\hline

\end{tabular}
\end{center}

\begin{center}
\begin{tabular}{ |c|c| } 
\hline
Quandle, X & Dis($X$) \\
\hline
&\\

\multirow{4}{6em}{$\begin{bmatrix}
    1 & 1 & 1 & 1 \\
    2 & 2 & 2 & 2 \\
    3 & 3 & 3 & 3 \\
    4 & 4 & 4 & 4 \\
\end{bmatrix}$} & \\ & $\{ 1 \}$ \\ & \\ & \\ & \\
\hline
&\\

\multirow{4}{6em}{$\begin{bmatrix}
    1 & 1 & 1 & 1 \\
    2 & 2 & 2 & 2 \\
    3 & 4 & 3 & 3 \\
    4 & 3 & 4 & 4 \\ 
\end{bmatrix}$} & \\ & $\Z_2$ \\ & \\ & \\ & \\
\hline
&\\

\multirow{5}{6em}{$\begin{bmatrix}
    1 & 1 & 1 & 1 \\
    2 & 2 & 2 & 2 \\
    4 & 4 & 3 & 3 \\
    3 & 3 & 4 & 4 \\
\end{bmatrix}$} & \\ & $\Z_2$ \\ & \\ & \\ & \\
\hline
&\\

\multirow{5}{6em}{$\begin{bmatrix}
    1 & 1 & 1 & 1 \\
    3 & 2 & 2 & 2 \\
    4 & 3 & 3 & 3 \\
    2 & 4 & 4 & 4 \\  
\end{bmatrix}$} & \\ & $\Z_3$ \\ & \\ & \\ & \\
\hline
&\\

\multirow{5}{6em}{$\begin{bmatrix}
    1 & 1 & 2 & 2 \\
    2 & 2 & 1 & 1 \\
    4 & 4 & 3 & 3 \\
    3 & 3 & 4 & 4 \\
\end{bmatrix}$} & \\ & $\Z_2$ \\ & \\ & \\ & \\
\hline
&\\

\multirow{5}{6em}{$\begin{bmatrix}
    1 & 1 & 1 & 1 \\
    2 & 2 & 4 & 3 \\
    3 & 4 & 3 & 2 \\
    4 & 3 & 2 & 4 \\
\end{bmatrix}$} & \\ & $D_3$ \\ & \\ & \\ & \\
\hline
&\\

\multirow{5}{6em}{$\begin{bmatrix}
    1 & 4 & 2 & 3 \\
    3 & 2 & 4 & 1 \\
    4 & 1 & 3 & 2 \\
    2 & 3 & 1 & 4 \\
\end{bmatrix}$} & \\ & $D_2$ \\ & \\ & \\ & \\
\hline

\end{tabular}

\end{center}


\begin{center}
\begin{tabular}{ |c|c| } 
\hline
Quandle, X & Dis($X$) \\
\hline
&\\

\multirow{5}{7em}{$\begin{bmatrix}
1 & 1 & 1 & 1 & 1\\
2 & 2 & 2 & 2 & 2\\
3 & 3 & 3 & 3 & 3\\
4 & 4 & 4 & 4 & 4\\
5 & 5 & 5 & 5 & 5\\
\end{bmatrix}$} & \\ & $\{ 1 \}$ \\ & \\ & \\ & \\ & \\
\hline
&\\

\multirow{5}{7em}{$\begin{bmatrix}
1 & 1 & 1 & 1 & 2\\
2 & 2 & 2 & 2 & 1\\
3 & 3 & 3 & 3 & 3\\
4 & 4 & 4 & 4 & 4\\
5 & 5 & 5 & 5 & 5\\
\end{bmatrix}$} & \\ & $\Z_2$ \\ & \\ & \\ & \\ & \\
\hline
&\\

\multirow{5}{7em}{$\begin{bmatrix}
1 & 1 & 1 & 2 & 2\\
2 & 2 & 2 & 1 & 1\\
3 & 3 & 3 & 3 & 3\\
4 & 4 & 4 & 4 & 4\\
5 & 5 & 5 & 5 & 5\\
\end{bmatrix}$} & \\ & $\Z_2$ \\ & \\ & \\ & \\ & \\
\hline
&\\

\multirow{5}{7em}{$\begin{bmatrix}
1 & 1 & 2 & 2 & 2\\
2 & 2 & 1 & 1 & 1\\
3 & 3 & 3 & 3 & 3\\
4 & 4 & 4 & 4 & 4\\
5 & 5 & 5 & 5 & 5\\
\end{bmatrix}$} & \\ & $\Z_2$ \\ & \\ & \\ & \\ & \\
\hline
&\\

\multirow{5}{7em}{$\begin{bmatrix}
1 & 1 & 1 & 1 & 1\\
3 & 2 & 2 & 2 & 2\\
2 & 3 & 3 & 3 & 3\\
5 & 4 & 4 & 4 & 4\\
4 & 5 & 5 & 5 & 5\\
\end{bmatrix}$} & \\ & $\Z_2$ \\ & \\ & \\ & \\ & \\
\hline
&\\

\multirow{5}{7em}{$\begin{bmatrix}
1 & 1 & 1 & 1 & 1\\
2 & 2 & 2 & 2 & 2\\
3 & 4 & 3 & 3 & 3\\
4 & 5 & 4 & 4 & 4\\
5 & 3 & 5 & 5 & 5\\
\end{bmatrix}$} & \\ & $\Z_3$ \\ & \\ & \\ & \\ & \\
\hline

\end{tabular}
\end{center}

\pagebreak

\begin{center}
\begin{tabular}{ |c|c| } 
\hline
Quandle, X & Dis($X$) \\
\hline
&\\

\multirow{5}{7em}{$\begin{bmatrix}
1 & 1 & 1 & 1 & 1\\
2 & 2 & 2 & 2 & 2\\
4 & 4 & 3 & 3 & 3\\
5 & 5 & 4 & 4 & 4\\
3 & 3 & 5 & 5 & 5\\
\end{bmatrix}$} & \\ & $\Z_3$ \\ & \\ & \\ & \\ & \\
\hline
&\\

\multirow{5}{7em}{$\begin{bmatrix}
1 & 1 & 1 & 1 & 1\\
2 & 2 & 2 & 2 & 2\\
4 & 5 & 3 & 3 & 3\\
5 & 3 & 4 & 4 & 4\\
3 & 4 & 5 & 5 & 5\\
\end{bmatrix}$} & \\ & $\Z_3$ \\ & \\ & \\ & \\ & \\
\hline
&\\

\multirow{5}{7em}{$\begin{bmatrix}
1 & 1 & 1 & 1 & 1\\
2 & 2 & 2 & 3 & 3\\
3 & 3 & 3 & 2 & 2\\
4 & 5 & 5 & 4 & 4\\
5 & 4 & 4 & 5 & 5\\
\end{bmatrix}$} & \\ & $D_2$ \\ & \\ & \\ & \\ & \\
\hline
&\\

\multirow{5}{7em}{$\begin{bmatrix}
1 & 1 & 1 & 1 & 1\\
2 & 2 & 2 & 3 & 3\\
3 & 3 & 3 & 2 & 2\\
5 & 4 & 4 & 4 & 4\\
4 & 5 & 5 & 5 & 5\\
\end{bmatrix}$} & \\ & $D_2$ \\ & \\ & \\ & \\ & \\
\hline
&\\

\multirow{5}{7em}{$\begin{bmatrix}
1 & 1 & 1 & 1 & 1\\
2 & 2 & 2 & 3 & 3\\
3 & 3 & 3 & 2 & 2\\
5 & 5 & 5 & 4 & 4\\
4 & 4 & 4 & 5 & 5\\
\end{bmatrix}$} & \\ & $\Z_2$ \\ & \\ & \\ & \\ & \\
\hline
&\\

\multirow{5}{7em}{$\begin{bmatrix}
1 & 1 & 1 & 1 & 1\\
3 & 2 & 2 & 3 & 3\\
2 & 3 & 3 & 2 & 2\\
5 & 4 & 4 & 4 & 4\\
4 & 5 & 5 & 5 & 5\\
\end{bmatrix}$} & \\ & $D_2$ \\ & \\ & \\ & \\ & \\
\hline

\end{tabular}
\end{center}

\pagebreak

\begin{center}
\begin{tabular}{ |c|c| } 
\hline
Quandle, X & Dis($X$) \\
\hline
&\\

\multirow{5}{7em}{$\begin{bmatrix}
1 & 1 & 1 & 1 & 1\\
3 & 2 & 2 & 3 & 3\\
2 & 3 & 3 & 2 & 2\\
5 & 5 & 5 & 4 & 4\\
4 & 4 & 4 & 5 & 5\\
\end{bmatrix}$} & \\ & $D_2$ \\ & \\ & \\ & \\ & \\
\hline
&\\

\multirow{5}{7em}{$\begin{bmatrix}
1 & 1 & 1 & 1 & 1\\
4 & 2 & 2 & 2 & 2\\
5 & 3 & 3 & 3 & 3\\
3 & 4 & 4 & 4 & 4\\
2 & 5 & 5 & 5 & 5\\
\end{bmatrix}$} & \\ & $\Z_4$ \\ & \\ & \\ & \\ & \\
\hline
&\\

\multirow{5}{7em}{$\begin{bmatrix}
1 & 1 & 1 & 1 & 1\\
2 & 2 & 2 & 2 & 2\\
3 & 3 & 3 & 5 & 4\\
4 & 4 & 5 & 4 & 3\\
5 & 5 & 4 & 3 & 5\\
\end{bmatrix}$} & \\ & $D_3$ \\ & \\ & \\ & \\ & \\
\hline
&\\

\multirow{5}{7em}{$\begin{bmatrix}
1 & 1 & 2 & 2 & 2\\
2 & 2 & 1 & 1 & 1\\
4 & 4 & 3 & 3 & 3\\
5 & 5 & 4 & 4 & 4\\
3 & 3 & 5 & 5 & 5\\
\end{bmatrix}$} & \\ & $\Z_6$ \\ & \\ & \\ & \\ & \\
\hline
&\\

\multirow{5}{7em}{$\begin{bmatrix}
1 & 1 & 2 & 2 & 2\\
2 & 2 & 1 & 1 & 1\\
3 & 3 & 3 & 5 & 4\\
4 & 4 & 5 & 4 & 3\\
5 & 5 & 4 & 3 & 5\\
\end{bmatrix}$} & \\ & $D_3$ \\ & \\ & \\ & \\ & \\
\hline
&\\

\multirow{5}{7em}{$\begin{bmatrix}
1 & 1 & 2 & 2 & 2\\
2 & 2 & 1 & 1 & 1\\
4 & 5 & 3 & 5 & 4\\
5 & 3 & 5 & 4 & 3\\
3 & 4 & 4 & 3 & 5\\
\end{bmatrix}$} & \\ & $D_3$ \\ & \\ & \\ & \\ & \\
\hline

\end{tabular}
\end{center}

\pagebreak

\begin{center}
\begin{tabular}{ |c|c| } 
\hline
Quandle, X & Dis($X$) \\
\hline
&\\

\multirow{5}{7em}{$\begin{bmatrix}
1 & 3 & 5 & 2 & 4\\
5 & 2 & 4 & 1 & 3\\
4 & 1 & 3 & 5 & 2\\
3 & 5 & 2 & 4 & 1\\
2 & 4 & 1 & 3 & 5\\
\end{bmatrix}$} & \\ & $\Z_5$ \\ & \\ & \\ & \\ & \\
\hline
&\\

\multirow{5}{7em}{$\begin{bmatrix}
1 & 1 & 1 & 1 & 1\\
2 & 2 & 5 & 3 & 4\\
3 & 4 & 3 & 5 & 2\\
4 & 5 & 2 & 4 & 3\\
5 & 3 & 4 & 2 & 5\\
\end{bmatrix}$} & \\ & $A_4$ \\ & \\ & \\ & \\ & \\
\hline
&\\

\multirow{5}{7em}{$\begin{bmatrix}
1 & 5 & 4 & 3 & 2\\
3 & 2 & 1 & 5 & 4\\
5 & 4 & 3 & 2 & 1\\
2 & 1 & 5 & 4 & 3\\
4 & 3 & 2 & 1 & 5\\
\end{bmatrix}$} & \\ & $\Z_5$ \\ & \\ & \\ & \\ & \\
\hline
&\\

\multirow{5}{7em}{$\begin{bmatrix}
1 & 4 & 2 & 5 & 3\\
4 & 2 & 5 & 3 & 1\\
2 & 5 & 3 & 1 & 4\\
5 & 3 & 1 & 4 & 2\\
3 & 1 & 4 & 2 & 5\\
\end{bmatrix}$} & \\ & $\Z_5$ \\ & \\ & \\ & \\ & \\
\hline

\end{tabular}
\end{center} 



\end{document}